\documentclass[11pt]{article}

\usepackage[T1]{fontenc}
\usepackage[utf8]{inputenc}
\usepackage{lmodern}
\usepackage[margin=1in]{geometry}
\usepackage{microtype}
\usepackage{amsmath,amssymb,amsthm,mathtools}
\usepackage{aliascnt}
\usepackage{booktabs,tabularx,array,longtable}
\usepackage{float}
\usepackage{graphicx}
\usepackage{tikz}
\usepackage{enumitem}
\usepackage{natbib}
\usepackage{xurl}
\usepackage[colorlinks=true,linkcolor=blue!55!black,citecolor=blue!55!black,urlcolor=blue!55!black]{hyperref}
\usepackage[nameinlink,noabbrev]{cleveref}

\hypersetup{
  pdftitle={Nonexistence of a Leech Tree of Order 18: A Computer-Assisted Proof},
  pdfauthor={Maseeh Ghodsi}
}

\crefname{theorem}{theorem}{theorems}
\Crefname{theorem}{Theorem}{Theorems}
\crefname{lemma}{lemma}{lemmas}
\Crefname{lemma}{Lemma}{Lemmas}
\crefname{proposition}{proposition}{propositions}
\Crefname{proposition}{Proposition}{Propositions}
\crefname{corollary}{corollary}{corollaries}
\Crefname{corollary}{Corollary}{Corollaries}
\crefname{definition}{definition}{definitions}
\Crefname{definition}{Definition}{Definitions}
\crefname{remark}{remark}{remarks}
\Crefname{remark}{Remark}{Remarks}
\crefname{section}{section}{sections}
\Crefname{section}{Section}{Sections}
\crefname{subsection}{section}{sections}
\Crefname{subsection}{Section}{Sections}
\crefname{figure}{figure}{figures}
\Crefname{figure}{Figure}{Figures}
\crefname{table}{table}{tables}
\Crefname{table}{Table}{Tables}
\crefname{appendix}{appendix}{appendices}
\Crefname{appendix}{Appendix}{Appendices}

\pdfmapfile{=lm.map}
\pdfmapfile{+symbols.map}

\usetikzlibrary{positioning,calc}

\newtheorem{theorem}{Theorem}[section]
\newaliascnt{proposition}{theorem}
\newtheorem{proposition}[proposition]{Proposition}
\aliascntresetthe{proposition}
\newaliascnt{lemma}{theorem}
\newtheorem{lemma}[lemma]{Lemma}
\aliascntresetthe{lemma}
\newaliascnt{corollary}{theorem}
\newtheorem{corollary}[corollary]{Corollary}
\aliascntresetthe{corollary}
\theoremstyle{definition}
\newaliascnt{definition}{theorem}

\aliascntresetthe{definition}
\newaliascnt{remark}{theorem}
\newtheorem{remark}[remark]{Remark}
\aliascntresetthe{remark}

\newcommand{\N}{\mathbb{N}}
\newcommand{\mex}{\operatorname{mex}_{+}}
\newcommand{\leanartifact}{\url{https://github.com/chesshippo/leech18-lean-artifact}}
\newcommand{\computationalartifact}{\url{https://github.com/chesshippo/leech18-computational-evidence}}
\newcommand{\sha}{\textsc{sha-256}}
\newcommand{\code}[1]{\texttt{\detokenize{#1}}}
\newsavebox{\reproducecommandbox}

\title{Nonexistence of a Leech Tree of Order 18:\\A Computer-Assisted Proof}
\author{Maseeh Ghodsi\\Independent Researcher\\\texttt{maseeh.ghodsi@gmail.com}}
\date{}

\hypersetup{pdfcreator={LaTeX with pdfTeX}}

\begin{document}
\maketitle

\begin{abstract}
A Leech tree of order $n$ is a tree with positive integral edge weights whose
$\binom n2$ pairwise weighted distances are precisely
$1,2,\ldots,\binom n2$.  We give a computer-assisted proof that no Leech tree
of order $18$ exists.  The argument has three layers.  First, a development in
Lean 4 verifies the structural facts used here.  These facts reduce every
putative example to one of eight local configurations and justify several
necessary conditions.
Second, conventional mathematical arguments prove a component-pair
whole-block exact-cover condition and the completeness of a recursive search.
Third, exhaustive computations close all eight configurations.  The
computation records exact coverage, source and input hashes, terminal receipts,
and independently checked exact-zero results.  The structural layer is
kernel-checked, but the search program, its execution, and the certificate
checker have not been formalized in Lean.  The result is therefore a
computer-assisted proof, not an end-to-end Lean proof.
\end{abstract}

\section{Introduction}

Let $T=(V,E,w)$ be a finite simple tree with vertex set $V$, edge set $E$,
and a strictly positive integral edge-weight function $w$.  The \emph{order}
of $T$ is its number of vertices.  The weighted distance between two vertices
is the sum of the edge weights on their unique connecting path.  Writing
$n=|V|$ and $N=\binom n2$, we call $T$ a \emph{Leech tree} when the
$N$ distances between unordered pairs of distinct vertices are exactly
\[
  1,2,\ldots,N,
\]
each occurring once.  Leech introduced this problem in 1975 and exhibited five
examples \citep{Leech1975}.  A weighted-tree isomorphism is a vertex bijection
that preserves adjacency and edge weights.  Up to this equivalence, the five
known examples have orders $2,3,4,4$, and $6$.  They are shown in
\cref{fig:known-examples}.

\begin{figure}[htbp]
\centering
\begin{tikzpicture}[
  vertex/.style={circle,fill=black,inner sep=1.8pt},
  edge label/.style={fill=white,inner sep=1.2pt,font=\small},
  every edge/.style={draw,thick}
]
  \begin{scope}[shift={(0,1.8)}]
    \node[font=\small] at (1,0.9) {$n=2$};
    \node[vertex] (a) at (0,0) {};
    \node[vertex] (b) at (2,0) {};
    \draw (a) -- node[edge label,above] {$1$} (b);
  \end{scope}
  \begin{scope}[shift={(4,1.8)}]
    \node[font=\small] at (1.5,0.9) {$n=3$};
    \node[vertex] (a) at (0,0) {};
    \node[vertex] (b) at (1.5,0) {};
    \node[vertex] (c) at (3,0) {};
    \draw (a) -- node[edge label,above] {$1$} (b)
              -- node[edge label,above] {$2$} (c);
  \end{scope}
  \begin{scope}[shift={(9,1.8)}]
    \node[font=\small] at (2.1,0.9) {$n=4$};
    \node[vertex] (a) at (0,0) {};
    \node[vertex] (b) at (1.4,0) {};
    \node[vertex] (c) at (2.8,0) {};
    \node[vertex] (d) at (4.2,0) {};
    \draw (a) -- node[edge label,above] {$1$} (b)
              -- node[edge label,above] {$3$} (c)
              -- node[edge label,above] {$2$} (d);
  \end{scope}
  \begin{scope}[shift={(2.1,-1.3)}]
    \node[font=\small] at (1.5,1.6) {$n=4$};
    \node[vertex] (c) at (1.5,0.4) {};
    \node[vertex] (a) at (0,0) {};
    \node[vertex] (b) at (1.5,-0.8) {};
    \node[vertex] (d) at (3,0) {};
    \draw (c) -- node[edge label,above left] {$1$} (a);
    \draw (c) -- node[edge label,right] {$2$} (b);
    \draw (c) -- node[edge label,above right] {$4$} (d);
  \end{scope}
  \begin{scope}[shift={(8,-1.3)}]
    \node[font=\small] at (2.6,1.6) {$n=6$};
    \node[vertex] (u) at (1.7,0.2) {};
    \node[vertex] (v) at (3.5,0.2) {};
    \node[vertex] (a) at (0.2,1.0) {};
    \node[vertex] (b) at (0.2,-0.6) {};
    \node[vertex] (c) at (5.0,1.0) {};
    \node[vertex] (d) at (5.0,-0.6) {};
    \draw (u) -- node[edge label,above] {$5$} (v);
    \draw (u) -- node[edge label,above left] {$1$} (a);
    \draw (u) -- node[edge label,below left] {$2$} (b);
    \draw (v) -- node[edge label,above right] {$4$} (c);
    \draw (v) -- node[edge label,below right] {$8$} (d);
  \end{scope}
\end{tikzpicture}
\caption{The five known Leech trees, up to weighted-tree isomorphism.  An edge
label is its positive integral weight.}
\label{fig:known-examples}
\end{figure}
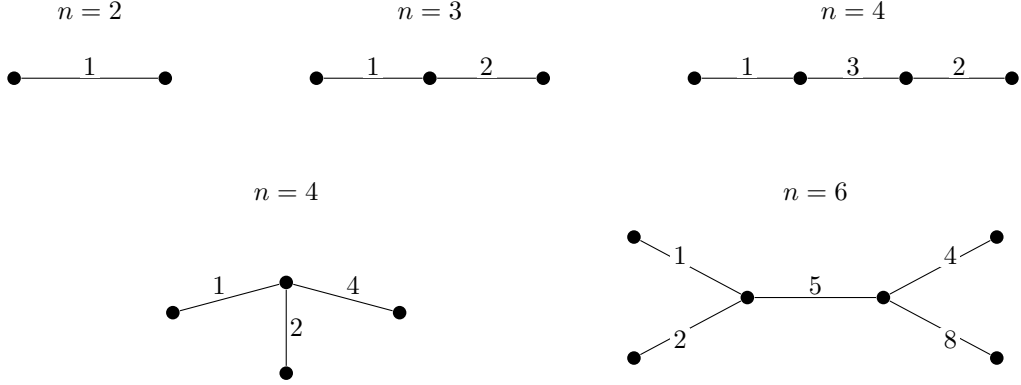

Taylor proved that the order of a Leech tree is a square or two more than a
square \citep{Taylor1977}.  Sz\'ekely, Wang, and Zhang used computation to
exclude orders $9$ and $11$ and also proved structural bounds on possible
Leech trees \citep{SzekelyWangZhang2005}.  Calhoun, Ferland, Lister, and
Polhill subsequently determined all perfect distance trees of order below
$18$, including the computational exclusion of order $16$, and used the same
least-missing-value recursion \citep{CalhounEtAl2007}.  Thus the forced
least-missing-weight strategy predates this work, and $18$ is the smallest
order not settled in the prior literature.  Varghese, Lakshmanan, and
Arumugam later determined all Leech trees of diameter three and excluded a
particular path-with-pendant family \citep{VargheseEtAl2020}.
At order $18$, $N=153$.

A direct enumeration of complete weighted trees is not feasible at this
order.  There are $123{,}867$ unlabelled tree shapes on $18$ vertices
\citep{OEISA000055}, and each has $17$ physical edges.  In a Leech tree, each
physical edge weight belongs to $I_{153}$, and the physical edge weights are
distinct.  A literal enumeration that assigns distinct values from $I_{153}$
to the edges of every unlabelled shape would therefore generate
\[
  123{,}867\,(153)_{17}
  =
  123{,}867\,\frac{153!}{136!}
  \approx 6.79\times10^{41}
\]
shape-weight assignments.  This is a baseline count for a direct
topology-by-weight enumeration, not a count of nonisomorphic weighted trees:
it includes symmetry-related assignments and assignments eliminated
immediately by elementary necessary conditions.  Even at a hypothetical rate
of $10^{18}$ complete candidates per second, examining this space would take
approximately $2.15\times10^{16}$ years.

Our main result is the following.

\begin{theorem}\label{thm:main}
There is no Leech tree on $18$ vertices.
\end{theorem}

The proof combines Lean-verified structure, conventional mathematics, and
finite computation.  The computation works with progressively exposed prefix
forests rather than with complete shape-weight assignments.  The first two
physical weights are forced to be $1$ and $2$.  Their two possible incidence
relations, together with the four possible incidences of the next forced edge,
give the eight exhaustive three-edge configurations studied in
\cref{sec:eight-configs}.  These configurations are intuitive, but their
exhaustiveness is proved before they are used as the roots of the computation.

At every subsequent prefix, the internal distances already realized inside
its components force the next physical edge weight: it is the least positive
distance not yet realized internally.  The search therefore chooses only the
two components and attachment vertices joined by that edge.  Exact
weighted-forest automorphisms group symmetric choices, and one representative
of each orbit is retained.

The principal mathematical look-ahead condition used by the computation is
the component-pair whole-block condition in \cref{thm:whole-block}.  In any
completion, every pair of current components accounts for all vertex pairs
having one endpoint in each component.  Their distances must occur together
as a complete additive block determined by the two attachment ports and the
route between them.  The blocks belonging to all component pairs must be
pairwise disjoint and must cover exactly the distances not yet realized inside
the prefix components.  This condition is necessary, not sufficient: passing
the test does not construct a Leech tree, while failure proves that the current
prefix has no Leech completion and the entire branch can be rejected.

Forced-weight recursion, exact orbit reduction, the whole-block condition,
and the remaining necessary tests reduce the problem to a finite search over
highly constrained prefix forests.  The resulting exhaustive computation
visited $8{,}567{,}320{,}605$ recursive nodes across the eight configurations.
This node count concerns partial recursive states, whereas the
$6.79\times10^{41}$ baseline concerns complete shape-weight assignments; the
two figures are not a formal speedup ratio.  Prefix-free partitions, terminal
receipts, and coverage checks certify the exhaustion of the surviving
branches.

The contribution here is an order-$18$ computer-assisted proof that combines the
structural reductions above with component-pair whole-block exact-cover
constraints, exact symmetry reduction, reproducible branch certificates, and
separate certificate checks.

A development in Lean 4 \citep{deMouraUllrich2021} verifies the initial
structural reduction and the physical-edge identities used by the search.  The
whole-block pruning theorem and the search-completeness induction are proved
mathematically in this paper but are not formalized in Lean.  The C++
enumeration and its certificate checkers are also outside the Lean kernel.
Precise trust boundaries appear in \cref{sec:formal-status}.

\section{Preliminaries}\label{sec:preliminaries}

\subsection{Notation and basic terminology}

We use $\mathbb Z$ for the integers and
\[
  \N=\{0,1,2,\ldots\},
  \qquad
  \N_{>0}=\{1,2,3,\ldots\}.
\]
For a finite set $S$, the symbol $|S|$ denotes its cardinality.  For a set
$X$, we write
\[
  \binom{X}{2}=\bigl\{\{x,y\}:x,y\in X,\ x\ne y\bigr\}
\]
for the set of unordered pairs of distinct elements of $X$.  Thus
$\binom n2=n(n-1)/2$ is the number of such pairs when $|X|=n$.  For sets
$A$ and $B$, $A\times B$ denotes their Cartesian product.  All set unions,
intersections, and differences have their usual meanings.

Throughout the paper, $T=(V,E,w)$ denotes a finite simple tree with
$w:E\to\N_{>0}$.  Its order is $n=|V|$, and
\[
  N=\binom n2,
  \qquad
  I_N=\{1,2,\ldots,N\}.
\]
The set $I_N$ is called the \emph{target interval}.  For vertices $x,y\in V$,
let $P_{xy}\subseteq E$ be the edge set of the unique simple path between
them, and define
\[
  d_T(x,y)=\sum_{e\in P_{xy}}w(e).
\]
The \emph{distance spectrum} is the multiset of values $d_T(x,y)$ indexed by
$\{x,y\}\in\binom V2$.  The tree is a Leech tree precisely when this indexed
family is a bijection from $\binom V2$ to $I_N$.

An edge of $T$ is called a \emph{physical edge} when it must be distinguished
from an edge of a quotient or search tree.  We write $e=ab$ to say that $a$
and $b$ are the endpoints of the physical edge $e$.  Since a physical edge is
itself a vertex-pair path, all physical weights of a Leech tree lie in $I_N$
and are distinct.

For $q\in\N_{>0}$, the \emph{smaller-weight prefix forest at $q$}, denoted
$F_{<q}$, has vertex set $V$ and precisely the physical edges $e$ satisfying
$w(e)<q$.  An edge in $F_{<q}$ is \emph{exposed}; a physical edge not yet in
$F_{<q}$ is \emph{unexposed}.  Vertices not incident with an exposed edge are
singleton components.  If $K$ is a component of a prefix forest, then
$d_K(x,y)$ denotes the weighted distance between $x,y\in K$ using the exposed
edges of $K$.  An \emph{internal prefix distance} is a value $d_K(x,y)$ for
two distinct vertices in the same component $K$.

\begin{lemma}[Prefix components are subtrees]\label{lem:prefix-subtree}
Let $T$ be a tree and $F\subseteq E(T)$.  Each component $K$ of the spanning
subforest $(V,F)$ is a subtree of $T$, and for all $x,y\in K$ the unique
$x$--$y$ path in $T$ lies in $K$; consequently
$d_K(x,y)=d_T(x,y)$.
\end{lemma}

\begin{proof}
The component $K$ is connected, so it contains an $x$--$y$ path.  Since $T$
is a tree, that path is the unique $x$--$y$ path in $T$.
\end{proof}

A prefix forest is called \emph{valid} if all of its internal distances,
taken together across all components, are pairwise distinct and belong to
$I_N$.  A \emph{completion} of a prefix is a Leech tree on the same vertex
set that contains every exposed edge with its displayed weight.  For a set
$D\subseteq\mathbb Z$, the \emph{positive minimum excluded value} is
\[
  \mex(D)=\min\bigl(\N_{>0}\setminus D\bigr).
\]
Thus $\operatorname{mex}_{+}(D)$ is the least strictly positive integer not in
$D$; the subscript $+$ distinguishes it from a convention in which the
minimum excluded value may be $0$.

\begin{lemma}[First physical weights]\label{lem:first-weights}
If $n\ge 3$, the physical weights $1$ and $2$ each occur exactly once.
\end{lemma}

\begin{proof}
The distance $1$ must be realized by a path.  Positivity forces that path to
be one physical edge of weight $1$.  A path of total weight $2$ is either one
edge of weight $2$ or two edges of weight $1$.  The latter is impossible
because physical weights are distinct.  Hence weight $2$ also occurs as a
physical edge, and uniqueness follows from uniqueness of every distance.
\end{proof}

\begin{lemma}[Forced least missing distance]\label{lem:forced-mex}
Let $T$ be a Leech tree of order $n$, and let $F\subseteq E(T)$ be any set of
physical edges.  Write $D(F)$ for the set of internal distances in the
spanning subforest $(V,F)$, and put $m=\mex(D(F))$.
\begin{enumerate}[label=(\roman*)]
  \item\label{item:forced-mex-forward} If $e$ is a physical edge of weight
  $q$ and $F=F_{<q}$, then $\mex(D(F))=q$.
  \item\label{item:forced-mex-converse} If $m\le N$, then $T$ has an edge of
  weight exactly $m$.  That edge is unique, lies outside $F$, and is a
  minimum-weight edge of $T$ among those outside $F$.
\end{enumerate}
\end{lemma}

\begin{proof}
For \ref{item:forced-mex-forward}, the value $q$ is not in $D(F)$, since it is
already the distance between the endpoints of $e$ and all distances in a
Leech tree are unique.  Now fix $r<q$.  The distance $r$ occurs in the final
tree.  Every edge on the path realizing $r$ has positive weight at most $r$,
hence has weight less than $q$.  That path is therefore internal to one prefix
component, so $r\in D(F)$.

For \ref{item:forced-mex-converse}, bijectivity gives distinct vertices
$x,y$ with $d_T(x,y)=m$; let $P$ be their path.  Every proper subpath of $P$
has weight less than $m$, so its weight lies in $D(F)$.  Some pair of vertices
in one component of $(V,F)$ therefore has the same distance.  Distance
uniqueness identifies that pair with the endpoints of the subpath, and
\cref{lem:prefix-subtree} then places the entire subpath in $F$.  If $P$ had
$k\ge2$ edges $v_0v_1,\ldots,v_{k-1}v_k$, applying this observation to the
two proper subpaths from $v_0$ to $v_{k-1}$ and from $v_1$ to $v_k$ would
place every edge of $P$ in $F$.  This would give $m\in D(F)$, contrary to the
definition of $m$.  Thus $P$ consists of one physical edge of weight $m$,
which necessarily lies outside $F$.

For minimality, suppose $e'\notin F$ had $w(e')<m$.  Then $w(e')\in D(F)$,
so some pair inside a component of $F$ realizes it.  By bijectivity that pair
is the endpoint pair of $e'$, and \cref{lem:prefix-subtree} forces $e'\in F$,
a contradiction.  Uniqueness follows from distinctness of distances.
\end{proof}

The Lean development verifies both preceding lemmas in their full formal
setting, together with the following proposition.

\begin{proposition}[Persistent merge block]\label{prop:persistent-block}
Let a physical edge $e=ab$ of a Leech tree $T$, with weight $q$, join two
components $A$ and $B$ of the smaller-weight prefix.  The indexed values
\[
  d_A(x,a)+q+d_B(b,y),\qquad (x,y)\in A\times B,
\]
are pairwise distinct, lie in $\{1,\ldots,N\}$, avoid all preceding internal
distances, and remain internal distances at every later prefix.
\end{proposition}

\begin{proof}
By \Cref{lem:prefix-subtree}, the paths within $A$ and $B$ are their paths in
$T$.  Thus each displayed sum is the distance $d_T(x,y)$, because the
unique path from $x$ to $y$ uses $e$.  The Leech property gives injectivity,
range, and disjointness from all other indexed pairs.  Once $e$ is exposed,
$x$ and $y$ remain in the same prefix component, which proves persistence.
\end{proof}

\subsection{Taylor's order and parity restriction}

Taylor proved that the order $n$ of a Leech tree must have the form $k^2$ or
$k^2+2$ for some $k\in\N_{>0}$ \citep{Taylor1977}.  The accompanying Lean
development includes a formalization of this result.

For the order-$18$ parity condition used by our search, fix any root $o\in V$
and partition $V$ according to whether $d_T(o,v)$ is even or odd; write the
class sizes as $a$ and $b$, so $a+b=18$.  If $m$ is the meet of $x$ and $y$
in the tree rooted at $o$, then
\[
 d_T(x,y)=d_T(o,x)+d_T(o,y)-2d_T(o,m)
 \equiv d_T(o,x)+d_T(o,y)\pmod 2.
\]
Thus $d_T(x,y)$ is odd precisely when $x$ and $y$ lie in different classes,
so the number of odd distances is $ab$.  The odd values in $I_{153}$ number
$(153+1)/2=77$, and bijectivity gives $ab=77$.  Together with $a+b=18$, this
gives $t^2-18t+77=0$, whose discriminant is $324-308=16$; hence
$\{a,b\}=\{7,11\}$.  Consistently, the number of even pairs is
$\binom72+\binom{11}2=21+55=76$.  Changing the root preserves the partition
or swaps its classes, since
\[
 d_T(o,v)\equiv d_T(o,o')+d_T(o',v)\pmod 2.
\]
The unordered pair of class sizes is therefore root-independent.  We call
the search condition requiring this pair the \emph{$7/11$ parity condition}.

\begin{remark}
This is the order-$18$ case of Taylor's parity restriction
\citep{Taylor1977}.  The order restriction and its order-$18$ parity
consequence are formalized in the accompanying Lean development.
\end{remark}

The \emph{hop length} of a path is its number of physical edges, and the
maximum hop length over all simple paths is the \emph{hop diameter}.

\begin{lemma}[Hop-diameter bound]\label{lem:hop-diameter}
Every simple path in an order-$18$ Leech tree contains at most $14$ physical
edges.
\end{lemma}

\begin{proof}
Suppose that a simple path has at least $15$ edges.  Choose a consecutive
$15$-edge subpath and write its edge weights in path order as
\[
  a_1,a_2,\ldots,a_{15}.
\]
Consider all consecutive subpaths containing one, two, or three edges.  Their
number is
\[
  15+14+13=42.
\]
These subpaths correspond to $42$ distinct unordered pairs of vertices.
Their weights are therefore $42$ distinct positive integers by the Leech
property.  If $S$ denotes their sum, then
\[
  S\ge 1+2+\cdots+42=903.
\]

We next count the contribution of each edge weight to $S$.  Every edge not
among the first two or last two occurs in six selected subpaths: one of length
one, two of length two, and three of length three.  The first, second,
fourteenth, and fifteenth edges have respective deficits $3,1,1,3$ from this
multiplicity.  Hence
\[
  S
  =6\sum_{i=1}^{15}a_i
   -\left(3a_1+a_2+a_{14}+3a_{15}\right).
\]
The complete $15$-edge subpath is itself a vertex-pair path, so its weight
belongs to $I_{153}$ and
\[
  \sum_{i=1}^{15}a_i\le153.
\]
The four boundary weights $a_1,a_2,a_{14},a_{15}$ are distinct positive
integers.  The expression with coefficients $3,1,1,3$ is minimized by placing
$1$ and $2$ at the coefficient-$3$ positions and $3$ and $4$ at the
coefficient-$1$ positions.  Therefore
\[
  3a_1+a_2+a_{14}+3a_{15}\ge16.
\]
Consequently,
\[
  S\le6\cdot153-16=902,
\]
contradicting $S\ge903$.  Thus no simple path contains $15$ edges.  Any longer
path contains a consecutive $15$-edge subpath, so every simple path has at
most $14$ edges.
\end{proof}

This is also the route formalized in Lean.  The contradiction has margin
exactly one: the lower bound is $903$, whereas the upper bound is $902$.

\section{The eight initial configurations}\label{sec:eight-configs}

Let $e_j$ denote the physical edge of weight $j$, when it exists.  Two
physical edges are \emph{adjacent} if they share a vertex.  The smaller-weight
prefix after three physical edges has one of the eight forms in
\cref{fig:eight-configs}.  The pictured vertices and components are only the
exposed local forest.  Any pictured vertex may receive further incident edges
in a completion.

The eight configurations can be seen intuitively before doing any calculation.
Draw $e_1$ and $e_2$ either adjacent or disjoint.  In the adjacent case the
next forced physical edge is $e_4$, while in the disjoint case it is $e_3$.
That new edge can meet neither exposed edge, just the first, just the second,
or both.  The two initial choices and four attachment choices give
$2\cdot4=8$ pictures.  We nevertheless state and prove the classification
explicitly because these eight cases are the exhaustive seeds of the
computation.

\begin{figure}[htbp]
\centering
\begin{tikzpicture}[
  scale=0.96,
  vertex/.style={circle,fill=black,inner sep=1.7pt},
  edge label/.style={fill=white,inner sep=1pt,font=\small},
  title/.style={font=\small\bfseries},
  every edge/.style={draw,thick}
]
  \begin{scope}[shift={(0,9)}]
    \node[title] at (2.2,1.0) {Configuration 1};
    \node[vertex] (a) at (0,0) {}; \node[vertex] (b) at (1.3,0) {};
    \node[vertex] (c) at (2.6,0) {}; \node[vertex] (d) at (3.5,0) {};
    \node[vertex] (e) at (4.8,0) {};
    \draw (a)--node[edge label,above]{$1$}(b)--node[edge label,above]{$2$}(c);
    \draw (d)--node[edge label,above]{$4$}(e);
  \end{scope}
  \begin{scope}[shift={(7.5,9)}]
    \node[title] at (2.0,1.0) {Configuration 2};
    \node[vertex] (a) at (0,0) {}; \node[vertex] (b) at (1.3,0) {};
    \node[vertex] (c) at (2.6,0) {}; \node[vertex] (d) at (3.9,0) {};
    \draw (a)--node[edge label,above]{$4$}(b)--node[edge label,above]{$1$}(c)
              --node[edge label,above]{$2$}(d);
  \end{scope}
  \begin{scope}[shift={(0,6)}]
    \node[title] at (2.0,1.0) {Configuration 3};
    \node[vertex] (a) at (0,0) {}; \node[vertex] (b) at (1.3,0) {};
    \node[vertex] (c) at (2.6,0) {}; \node[vertex] (d) at (3.9,0) {};
    \draw (a)--node[edge label,above]{$1$}(b)--node[edge label,above]{$2$}(c)
              --node[edge label,above]{$4$}(d);
  \end{scope}
  \begin{scope}[shift={(7.5,6)}]
    \node[title] at (2.0,1.0) {Configuration 4};
    \node[vertex] (c) at (2,0) {}; \node[vertex] (a) at (0.4,0.5) {};
    \node[vertex] (b) at (2,-1.0) {}; \node[vertex] (d) at (3.6,0.5) {};
    \draw (c)--node[edge label,above left]{$1$}(a);
    \draw (c)--node[edge label,right]{$2$}(b);
    \draw (c)--node[edge label,above right]{$4$}(d);
  \end{scope}
  \begin{scope}[shift={(0,3)}]
    \node[title] at (2.4,1.0) {Configuration 5};
    \node[vertex] (a) at (0,0) {}; \node[vertex] (b) at (1.1,0) {};
    \node[vertex] (c) at (1.8,0) {}; \node[vertex] (d) at (2.9,0) {};
    \node[vertex] (e) at (3.6,0) {}; \node[vertex] (f) at (4.7,0) {};
    \draw (a)--node[edge label,above]{$1$}(b);
    \draw (c)--node[edge label,above]{$2$}(d);
    \draw (e)--node[edge label,above]{$3$}(f);
  \end{scope}
  \begin{scope}[shift={(7.5,3)}]
    \node[title] at (2.2,1.0) {Configuration 6};
    \node[vertex] (a) at (0,0) {}; \node[vertex] (b) at (1.3,0) {};
    \node[vertex] (c) at (2.6,0) {}; \node[vertex] (d) at (3.5,0) {};
    \node[vertex] (e) at (4.8,0) {};
    \draw (a)--node[edge label,above]{$1$}(b)--node[edge label,above]{$3$}(c);
    \draw (d)--node[edge label,above]{$2$}(e);
  \end{scope}
  \begin{scope}[shift={(0,0)}]
    \node[title] at (2.2,1.0) {Configuration 7};
    \node[vertex] (a) at (0,0) {}; \node[vertex] (b) at (1.3,0) {};
    \node[vertex] (c) at (2.2,0) {}; \node[vertex] (d) at (3.5,0) {};
    \node[vertex] (e) at (4.8,0) {};
    \draw (a)--node[edge label,above]{$1$}(b);
    \draw (c)--node[edge label,above]{$2$}(d)--node[edge label,above]{$3$}(e);
  \end{scope}
  \begin{scope}[shift={(7.5,0)}]
    \node[title] at (2.0,1.0) {Configuration 8};
    \node[vertex] (a) at (0,0) {}; \node[vertex] (b) at (1.3,0) {};
    \node[vertex] (c) at (2.6,0) {}; \node[vertex] (d) at (3.9,0) {};
    \draw (a)--node[edge label,above]{$1$}(b)--node[edge label,above]{$3$}(c)
              --node[edge label,above]{$2$}(d);
  \end{scope}
\end{tikzpicture}
\caption{The eight exhaustive three-edge prefix forests.  The diagrams show
only exposed edges; further edges may be incident with any displayed vertex.}
\label{fig:eight-configs}
\end{figure}
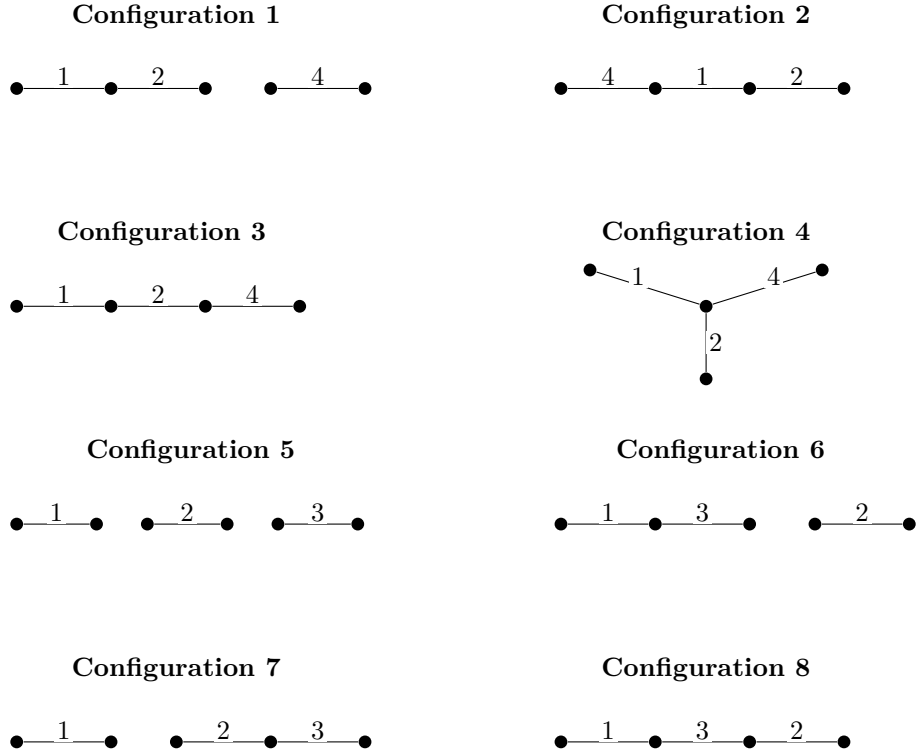

\begin{theorem}[Eight-configuration classification]\label{thm:eight-configs}
For $n\ge5$, the first three physical edges of a Leech tree have one of the
forms in \cref{tab:eight-configs}.  These eight forms are exhaustive.  The
hypothesis $n\ge5$ ensures $N\ge10$, so all least missing values through $7$
used below lie in $I_N$.
\end{theorem}

\begin{table}[htbp]
\centering
\small
\begin{tabularx}{\textwidth}{@{}c X l c@{}}
\toprule
Configuration & Three-edge prefix forest & Used distances & Next weight\\
\midrule
1 & path $1$--$2$ and a separate edge $4$ & $\{1,2,3,4\}$ & $5$\\
2 & path $4$--$1$--$2$ & $\{1,2,3,4,5,7\}$ & $6$\\
3 & path $1$--$2$--$4$ & $\{1,2,3,4,6,7\}$ & $5$\\
4 & three-edge star with weights $1,2,4$ & $\{1,2,3,4,5,6\}$ & $7$\\
5 & three separate edges with weights $1,2,3$ & $\{1,2,3\}$ & $4$\\
6 & path $1$--$3$ and a separate edge $2$ & $\{1,2,3,4\}$ & $5$\\
7 & separate edge $1$ and path $2$--$3$ & $\{1,2,3,5\}$ & $4$\\
8 & path $1$--$3$--$2$ & $\{1,2,3,4,5,6\}$ & $7$\\
\bottomrule
\end{tabularx}
\caption{Initial configurations, their internal distance sets, and the next
physical weight forced by part~\ref{item:forced-mex-converse} of
\cref{lem:forced-mex}.}
\label{tab:eight-configs}
\end{table}

\begin{proof}
By \Cref{lem:first-weights}, the physical edges $e_1$ and $e_2$ exist.  First
suppose they meet, say $e_1=ab$ and $e_2=bc$.  Their two-edge path realizes
distance $3$, so no physical edge of weight $3$ exists.  The least missing
distance is $4$.  Since $4\le N$ (indeed $N\ge6$ for $n\ge4$), part
\ref{item:forced-mex-converse} of \cref{lem:forced-mex} gives the physical
edge $e_4$.  Exhaustively, its adjacency set among the two exposed edges is
one of
\[
  \varnothing,\quad \{e_1\},\quad \{e_2\},\quad \{e_1,e_2\}.
\]
In this meeting case the endpoints of $e_1$ are not interchangeable, since
$b$ also lies on $e_2$.  Adjacency set $\{e_1\}$ therefore forces attachment
at $a$, since attaching at $b$ would place $e_2$ in the adjacency set;
likewise $\{e_2\}$ forces attachment at $c$.  For adjacency set
$\{e_1,e_2\}$, the chord $ac$ would close the triangle $abc$, so the shared
vertex is $b$ and the other endpoint is a new vertex.  Each adjacency set
thus determines a single forest.

The resulting forests are Configurations 1 through 4.  Summing weights along
their exposed paths gives respectively
\[
 \{1,2,3,4\},\quad \{1,2,3,4,5,7\},\quad
 \{1,2,3,4,6,7\},\quad \{1,2,3,4,5,6\},
\]
whose least positive missing values are $5,6,5,7$.

Now suppose $e_1=ab$ and $e_2=cd$ are disjoint.  The distance $3$ is not yet
internal.  Since $3\le N$ (again $N\ge6$), part
\ref{item:forced-mex-converse} of \cref{lem:forced-mex} gives the physical
edge $e_3$.  Its adjacency set among the two exposed edges has the same four
possibilities.  The exposed forest admits the automorphisms $a\leftrightarrow
b$ and $c\leftrightarrow d$.  Adjacency set $\{e_1\}$ therefore has two
placements, at $a$ or at $b$, which are exchanged by an automorphism and give
isomorphic forests; the same holds for $\{e_2\}$.  The bridging case
$\{e_1,e_2\}$ has four placements forming a single orbit.  Each adjacency set
again determines a single isomorphism class.  They give Configurations 5
through 8, with distance sets
\[
 \{1,2,3\},\quad \{1,2,3,4\},\quad \{1,2,3,5\},\quad
 \{1,2,3,4,5,6\},
\]
and next missing values $4,5,4,7$.  The physical-weight multiset together
with the internal distance set separates all eight configurations; in
particular, Configurations 1 and 6 both have internal distance set
$\{1,2,3,4\}$, but their physical-weight multisets are respectively
$\{1,2,4\}$ and $\{1,2,3\}$.  The four adjacency sets exhaust each group.
\end{proof}

The eight forms are pairwise non-isomorphic; only their exhaustiveness is used
below.

The Lean development verifies the full statement of
\cref{thm:eight-configs}, including the physical-edge status implicit in each
missing value.

\section{Necessary conditions for completion}\label{sec:conditions}

The recursive search maintains a weighted prefix forest with components
$K_1,\ldots,K_m$, where $m$ is the number of components and the subscripts are
arbitrary component indices.  For $x,y\in K_i$, write
$d_i(x,y)=d_{K_i}(x,y)$.  Let $D$ be the set of all internal prefix distances
over the $m$ components, and let
\[
  H=I_N\setminus D
\]
be the set of \emph{unfilled values}.  Here and below, $i<j$ means
$1\le i<j\le m$.  A \emph{port} of $K_i$ toward $K_j$ is the vertex of
$K_i$ at which the final path toward $K_j$ leaves that component.  The
\emph{route contribution} is the weighted length of the portion of that path
from the port in $K_i$ to the port in $K_j$, including the unexposed edges and
any intervening fixed segments.  The following theorem is the principal
pruning condition that made the exhaustive computation feasible.  Its proof
is conventional mathematics and has not yet been formalized in Lean.

\begin{theorem}[Component-pair whole-block exact cover]
\label{thm:whole-block}
Suppose a valid prefix forest is a weighted subforest of a Leech tree $T$,
meaning that every exposed edge has the same weight in the prefix and in $T$.
For every unordered component pair $i<j$, there are ports $p_{ij}\in K_i$ and
$p_{ji}\in K_j$ and a positive route contribution
$L_{ij}\in\N_{>0}$ such that the \emph{component-pair block}
\[
 B_{ij}=\bigl\{
 d_i(x,p_{ij})+L_{ij}+d_j(p_{ji},y):x\in K_i,\ y\in K_j
 \bigr\}
\]
satisfies all of the following:
\begin{enumerate}[label=(\roman*)]
  \item $|B_{ij}|=|K_i||K_j|$;
  \item $B_{ij}\subseteq H$;
  \item the blocks $B_{ij}$ are pairwise disjoint;
  \item $\bigcup_{i<j}B_{ij}=H$.
\end{enumerate}
\end{theorem}

\begin{proof}
By \Cref{lem:prefix-subtree}, each $K_i$ is a subtree of $T$, and the $K_i$
are pairwise vertex-disjoint.  There is a unique shortest subpath of $T$ that
joins $K_i$ to $K_j$ and has no internal vertex in either subtree.  Write its
endpoints as $p_{ij}\in K_i$ and $p_{ji}\in K_j$.  Uniqueness follows because
two distinct such connecting paths, together with paths inside the connected
subtrees $K_i$ and $K_j$, would form a cycle in $T$.

For arbitrary $x\in K_i$ and $y\in K_j$, concatenate the path in $K_i$ from
$x$ to $p_{ij}$, the connecting path from $p_{ij}$ to $p_{ji}$, and the path
in $K_j$ from $p_{ji}$ to $y$.  The three pieces meet only at their stated
endpoints, so the concatenation is the unique $x$--$y$ path in $T$.  Setting
$L_{ij}=d_T(p_{ij},p_{ji})$ therefore gives
\[
 d_T(x,y)=d_i(x,p_{ij})+L_{ij}+d_j(p_{ji},y).
\]
Positivity of $L_{ij}$ follows because the components are disjoint.

Different indexed vertex pairs in a Leech tree have different distances, so
the block is internally injective and has $|K_i||K_j|$ elements.  These values
belong to $\{1,\ldots,N\}$ and cannot equal a distance already internal to a
prefix component.  Hence $B_{ij}\subseteq H$.  Blocks associated with
different component pairs describe disjoint sets of indexed vertex pairs, so
their values are disjoint.  Finally,
\[
 \sum_{i<j}|K_i||K_j|
 =\binom n2-\sum_i\binom{|K_i|}{2}
 =|H|.
\]
The last equality uses the global pairwise distinctness in the definition of
a valid prefix forest, which gives
$|D|=\sum_i\binom{|K_i|}{2}$.
Thus the disjoint blocks fill $H$ exactly.
\end{proof}

\begin{corollary}[Disjoint blocks suffice]\label{cor:disjoint-blocks-suffice}
Given conditions \textup{(i)} and \textup{(ii)} of
\cref{thm:whole-block} and the counting identity in its proof, condition
\textup{(iii)} alone implies condition \textup{(iv)}.  Hence, to verify the
whole-block condition, it suffices to exhibit a pairwise-disjoint system
consisting of one candidate block per component pair.
\end{corollary}

\begin{proof}
The pairwise-disjoint blocks lie in $H$ and their cardinalities sum to $|H|$,
so their union is $H$.
\end{proof}

\begin{remark}\label{rem:not-converse}
The converse of \cref{thm:whole-block} is not claimed.  A block packing may
select ports and route lengths that cannot coexist in one tree.  Passing the
test therefore does not construct a completion.  Failing it proves that a
completion is impossible.
\end{remark}

For each component pair, the implementation enumerates all port pairs and all
positive translations whose complete additive block is injective and lies in
$H$.  The per-pair enumeration is exhaustive and needs no budget: route
contributions lie in $\{1,\ldots,153\}$ and port pairs number
$|K_i||K_j|\le81$: the components are disjoint and their total size is at
most $18$, so their product is maximized when both have size $9$.  Thus the
candidate count for one component pair is at most $153\cdot81=12{,}393$.
Consequently, a prefix is rejected if one component
pair has no candidate, and a cap-truncated enumeration is never treated as
empty.  At selected small component counts, the program also searches for a
pairwise-disjoint choice of one candidate per component pair.  Only this
global search is budgeted: exhaustive failure rejects the prefix, whereas
candidate-cap or search-budget exhaustion returns \code{UNKNOWN/PASS}, never
rejection.  Each filter is one-sided: it rejects only on a proof of
infeasibility.  A resource limit is not such a proof, so budget exhaustion
returns no information and the branch is retained.  Rejection on resource
exhaustion would not be sound.

On every invocation of the whole-block layer, the implementation checks the
counting identity
\[
  \sum_{i<j}|K_i||K_j|=|H|,
\]
together with the per-block conditions
\[
  |B_{ij}|=|K_i||K_j|
  \quad\text{and}\quad
  B_{ij}\subseteq H.
\]
A violation raises a fatal internal error.  These checks were active
throughout production runs totaling $8{,}567{,}320{,}605$ recursive node
visits, and no assertion failed.

The Lean development also proves a selected-edge gluing identity that places
\cref{thm:whole-block} in a polynomial form.  Let $z$ be an indeterminate and
define the \emph{distance polynomial} of $T$ by
\[
 P_T(z)=\sum_{\{x,y\}\in\binom V2}z^{d_T(x,y)}.
\]
If selected physical edges are deleted from $T$, leaving components $K_i$, let
\[
 P_i(z)=\sum_{\{x,y\}\in\binom{K_i}{2}}z^{d_i(x,y)}
\]
be the internal distance polynomial of $K_i$, and, for a port $p\in K_i$, let
\[
 R_{i,p}(z)=\sum_{x\in K_i}z^{d_i(x,p)}.
\]
This is the \emph{rooted depth polynomial} of $K_i$ at $p$.  For the extreme
ports on the route between $K_i$ and $K_j$, let $L_{ij}$ have the same meaning
as in \cref{thm:whole-block}.  Then
\begin{equation}\label{eq:gluing}
 P_T(z)=\sum_iP_i(z)+
 \sum_{i<j}z^{L_{ij}}R_{i,p_{ij}}(z)R_{j,p_{ji}}(z).
\end{equation}
For a Leech tree, the left side is $z+z^2+\cdots+z^N$.  Equality here is
\emph{coefficientwise}, meaning that the coefficients of every power of $z$
agree.  By \Cref{lem:prefix-subtree}, the terms $P_i$ record genuine
$T$-distances.  Since every term on the right of \cref{eq:gluing} has
nonnegative coefficients and their sum is $z+z^2+\cdots+z^N$, every $P_i$ and
every cross product has all coefficients in $\{0,1\}$, and their supports are
pairwise disjoint and partition $\{1,\ldots,N\}$.  This is a forward identity
for ports arising from a tree $T$, not a realization theorem for arbitrary
polynomials.

\begin{remark}
The coefficientwise identity \cref{eq:gluing} is the generating-function form
of \cref{thm:whole-block}; the two state the same exact-cover fact in different
language.
\end{remark}

\section{Exhaustive search and certificates}\label{sec:search}

\subsection{Recursive enumeration}\label{subsec:recursive-enumeration}

A \emph{search state} is a valid weighted prefix forest on all $18$ vertices,
together with its internal distance data.  Vertices not yet incident with an
exposed edge are singleton components.  A \emph{child} of a state is a valid
one-edge extension produced by the next forced weight.  At a nonterminal valid
state having a completion, at least two components remain, so the internal
distance set omits some value of $I_{153}$ and its least positive missing value
is at most $153$.  Part~\ref{item:forced-mex-converse} of
\cref{lem:forced-mex} therefore determines the next physical weight.  The
program then considers every pair of distinct components and every choice of
one attachment vertex in each.  Adding the forced-weight edge merges exactly
two components.  A \emph{weighted-forest automorphism} is a permutation of the
vertices that preserves adjacency and every exposed edge weight.  It acts on
candidate extensions, and an \emph{automorphism orbit} is an equivalence class
under this action.  The program retains one candidate edge from each exact
orbit, represented by canonical rooted and unrooted codes.

\begin{lemma}[Unique predecessor]\label{lem:unique-predecessor}
Every non-root search-state isomorphism class has a unique predecessor
isomorphism class.  Consequently, distinct branches of the search do not
reconverge after orbit reduction.
\end{lemma}

\begin{proof}
In a valid search state, every exposed edge weight is an internal distance
between the endpoints of that edge.  Validity therefore makes the exposed
edge weights pairwise distinct.  Moreover, the edges are added in strictly
increasing forced-weight order: if the current forced weight is
$m=\mex(D)$, then all positive values below $m$ already belong to $D$, and
after the new edge is added, $m$ also belongs to the new internal-distance
set.  The next positive missing value is thus greater than $m$.

Every non-root state consequently has a unique largest-weight exposed edge,
namely its most recently added edge.  Deleting that edge recovers its
predecessor.  Any weighted-forest isomorphism between two child states
preserves edge weights, so it carries the unique largest-weight edge of one
state to that of the other.  Deleting those two edges therefore induces an
isomorphism between the predecessors.  Hence the predecessor is well defined
on search-state isomorphism classes and is unique, so two distinct branch
histories cannot later merge into one isomorphism class.
\end{proof}

For every child, the program recomputes the complete new cross-distance block.
The frozen production search uses exactly five rejecting tests:
\begin{enumerate}[label=(\roman*)]
  \item two internal distances are equal;
  \item an internal distance lies outside $I_{153}$;
  \item the $7/11$ parity profile is impossible;
  \item an exposed path has more than $14$ physical edges;
  \item the component-pair whole-block necessary condition fails.
\end{enumerate}
For test~\textup{(iii)}, each current component has a bipartition determined
by distance parity, and the two classes may be swapped independently from one
component to another.  The test rejects only if no choice of these swaps gives
global class sizes $7$ and $11$.  Test~\textup{(iv)} rejects exactly when two
vertices in one current component are joined by an exposed path of hop length
greater than $14$.  In test~\textup{(v)}, rejection occurs if some component
pair has no candidate block, or if an exhaustive within-budget exact-cover
search proves that no pairwise-disjoint choice of one block per component pair
exists.  Candidate-cap or search-budget exhaustion returns
\code{UNKNOWN/PASS}.  Hall variants were disabled.  No experimental or shadow
test contributed to the result.

\begin{lemma}[Filter soundness]\label{lem:filter-soundness}
Let $F$ be a search state that is a weighted subforest of some order-$18$
Leech tree.  Then no filter of \cref{subsec:recursive-enumeration} rejects
$F$.
\end{lemma}

\begin{proof}
The distance-repeat and out-of-range tests do not trigger because, by
\cref{lem:prefix-subtree}, internal prefix distances are genuine distances of
the Leech tree, so they are distinct and belong to $I_{153}$.  The hop test
does not trigger by \cref{lem:hop-diameter}.  For the parity profile,
\cref{lem:prefix-subtree} shows that the two global
distance-parity classes restrict on each component to its two internal parity
classes, up to swapping them.  The completion therefore supplies a choice of
component swaps with class sizes $7$ and $11$.  Finally,
\cref{thm:whole-block} supplies a candidate block for every component pair and
a pairwise-disjoint exact cover, so neither form of whole-block rejection can
occur.  Resource exhaustion in that layer passes the state.  Thus none of the
five production tests rejects $F$.
\end{proof}

\begin{proposition}[Search completeness]\label{prop:search-complete}
Let $C$ be one of the eight three-edge configurations.  If some order-$18$
Leech tree has $C$, up to weighted-tree isomorphism, as its three
smallest-weight physical edges, then the search rooted at $C$ reaches an
accepting leaf.
\end{proposition}

\begin{proof}
We induct on the number of exposed edges, with the hypothesis that the current
state has at least one Leech completion.  At the root, this holds after
relabeling the tree supplied by the assumption.

Suppose the state $F$ has a Leech completion $T$.  Since a nonterminal forest
has at least two components and $F$ is valid, its least positive missing
internal distance $m$ satisfies $m\le153=N$.  By part
\ref{item:forced-mex-converse} of \cref{lem:forced-mex}, $m$ is the weight of
a physical edge $e$ of $T$ outside $F$, so the program's forced weight is
correct.  The endpoints of $e$ lie in distinct components of $F$; otherwise,
the path between them in that component together with $e$ would form a cycle
in $T$.  The endpoints are therefore among the component pairs and attachment
vertices enumerated by the program.  Hence some enumerated child $c$ satisfies
$c\subseteq T$, and $c$ has a Leech completion.

Orbit reduction retains a representative $c'$ weighted-forest isomorphic to
$c$.  Let $\tau$ be the corresponding permutation of the common vertex set.
Then $\tau(T)$ is a tree containing
$c'$ whose distance multiset equals that of $T$, so $\tau(T)$ is a Leech tree
and $c'$ has a Leech completion.  The inductive hypothesis therefore transfers
to the retained representative.  By \Cref{lem:filter-soundness}, $c'$ is not
pruned.

After $17$ physical edges, the forest is a tree on $18$ vertices with
$\binom{18}{2}=153$ vertex pairs.  If its $153$ distances are pairwise
distinct and all lie in $\{1,\ldots,153\}$, they exhaust that set, so the state
is a Leech tree and the leaf accepts.
\end{proof}

\begin{remark}
The proposition delivers a tree isomorphic to the hypothesized one, not the
tree itself; orbit reduction does not preserve vertex identities.  For a
nonexistence conclusion this is sufficient, since existence is invariant
under isomorphism.
\end{remark}

The proposition is a mathematical statement about the specified algorithm.
The implementation of canonicalization, recursion, and result reporting
remains in the trusted computational base.

\subsection{Partitioning and receipt discipline}\label{subsec:partitioning}

In this subsection, the \emph{fan-out} of a state means its
post-cross-distance canonical children, rather than its raw component-pair and
attachment choices.  A \emph{frontier} is the set of states at a fixed
recursion depth, where the depth is the number of exposed physical edges.  A
\emph{partition root} is a state chosen as the root of a separately executed
recursive subtree.  A set of partition roots is \emph{prefix-free} if no
selected root is a descendant of another selected root.

For a split state $s$, let $C(s)$ be the set of canonical children surviving
the mandatory cross-distance collision and range checks.  The
\code{child_max} value records $|C(s)|$, not the number of raw component-pair
and attachment choices.  Candidates excluded before \code{child_max} is
recorded fail necessary validity conditions and cannot contain a Leech
completion, by \cref{lem:filter-soundness}.  The children in $C(s)$ must be
divided into three disjoint classes: $S(s)$ consists of
children whose subtrees are covered by ordinary searched partition roots;
$Z(s)$ consists of children closed by separate exact-zero evidence; and
$R(s)$ consists of children rejected at the split by a proved necessary
condition.  Exhaustive coverage requires
\[
  C(s)=S(s)\mathbin{\dot\cup}Z(s)\mathbin{\dot\cup}R(s)
\]
at every split state.  A rejection placed in $R(s)$ is sound by
\cref{lem:filter-soundness}: a child having a Leech completion cannot fail any
of the implemented necessary conditions.

Prefix-freeness and coverage are different assertions.  Prefix-freeness says
only that the recorded partition pieces do not overlap.  Coverage is the
displayed equality: every canonical child must lead to searched roots, have
separate exact-zero evidence, or be soundly rejected.  Repeated work above a
partition root affects performance counters but neither assertion.

The released coverage verifier reads the \code{child_max} fields along all
$39{,}672$ receipt paths in the unified eight-configuration plan.  While a run
is following its prescribed path, this field gives the exact
post-cross-distance canonical fan-out of each traversed split state.  The
verifier reconstructs a prefix trie for each of the nine solver modes,
checks for conflicting or unknown fan-outs, children outside the measured
range, and nodes that are simultaneously internal and terminal, and compares
the recorded child indices with $\{0,\ldots,f-1\}$ at every split of measured
fan-out $f$.  It also requires every planned leaf to have an exact
\code{ZERO} receipt with matching raw output and empty standard error.

The plan labels $39{,}178$ receipts as \code{primary_cover}, $464$ as
\code{coverage_completion}, and $30$ as
\code{direct_replay_completion}.  These labels record how the final cover was
assembled; all $39{,}672$ entries are ordinary uncapped terminal searches in
the released package.  The verifier proves prefix-freeness and complete
fan-out coverage separately.  It does not execute a pruning filter, and it
reads the \code{child_max} values produced by the search program rather than
independently regenerating the canonical children or verifying the orbit
canonicalization.

Each terminal solver process recomputes the least positive missing internal
distance of its three-edge seed and compares it with the value recorded for
that configuration.  A mismatch terminates with an internal-error status and
produces no result record.

A \emph{search piece} is the recursive subtree below one partition root.  Its
\emph{receipt} is the recorded terminal metadata binding the selected prefix,
program arguments, hashes, counters, exit status, and result.  The status
\code{ZERO} means that the entire selected subtree terminated with zero
accepting leaves and zero abandoned branches.  A search piece closed a branch
only when it produced a terminal exact \code{ZERO} result with exit code $0$,
the expected arguments and hashes, a positive node count, zero solutions, and
complete receipt metadata.  A timeout, resource limit, missing receipt,
nonzero exit, abnormal termination, or incomplete wrapper record counted as
non-evidence.  Such pieces were excluded or replaced by complete child covers.

The exact-cover layer also failed open.  Its $100$-state local budget and its
candidate cap could only return \code{UNKNOWN/PASS}.  The recorded production
arguments disable both Hall variants, so no Hall check contributes to the
zero result.  These policies separate an exhaustive branch-closing
computation from a heuristic search.

\section{Computational results}\label{sec:results}

\Cref{tab:results} gives the complete exhaustive result.

\begin{table}[htbp]
\centering
\small
\begin{tabular}{@{}rrr l@{}}
\toprule
Configuration & Receipts & Node visits & Result\\
\midrule
1 & 5,202 & 1,321,606,322 & exact zero\\
2 & 79 & 193,281,350 & exact zero\\
3 & 47 & 167,742,832 & exact zero\\
4 & 1,324 & 227,134,788 & exact zero\\
5 & 25,684 & 4,242,085,009 & exact zero\\
6 & 3,983 & 1,165,724,556 & exact zero\\
7 & 3,301 & 1,010,043,695 & exact zero\\
8 & 52 & 239,702,053 & exact zero\\
\midrule
\textbf{Total} & \textbf{39,672} & \textbf{8,567,320,605} &\\
\bottomrule
\end{tabular}
\caption{Exhaustive computations for all eight configurations in the unified
release plan.  The plan contains $39{,}672$ exact-zero terminal receipts.
Node visits include repeated prefix work and are not globally unique states.}
\label{tab:results}
\end{table}

\begin{table}[htbp]
\centering
\small
\begin{tabularx}{\textwidth}{@{}l l X@{}}
\toprule
Plan class & Receipts & Role in the final cover\\
\midrule
\code{primary_cover} & 39,178 & Primary prefix-free partition leaves\\
\code{coverage_completion} & 464 & Additional leaves required by measured fan-out\\
\code{direct_replay_completion} & 30 & Direct terminal searches completing Configuration~4\\
\bottomrule
\end{tabularx}
\caption{Coverage classes in the clean execution plan.  Every row represents
uncapped terminal searches with exact-zero receipts.}
\label{tab:evidence-groups}
\end{table}

The unified plan contains $39{,}672$ search receipts arranged in $208$
bundles.  Configuration~3 is represented directly by seven leaves in mode
\code{a2_attached} and forty leaves in mode \code{a2_separate}.  The other
configuration totals are shown in \cref{tab:results}.  Every terminal solver
run had no node, depth, or stop-depth cap and ended with zero accepting leaves
and zero abandoned branches.  The released verifier checks every receipt,
raw output, empty error stream, bundle binding, and the complete nine-mode
coverage trie.  No external certificate or separate evidence package is
needed for any configuration.  See \Cref{app:reproducibility}.

\begin{proposition}[Certified exhaustion]\label{prop:certified-exhaustion}
For each of the eight initial configurations, every canonical search branch
is either exhausted by a terminal search, closed by explicit exact-zero
evidence, or rejected by a proved necessary condition.  No accepting leaf or
abandoned branch remains.
\end{proposition}

\begin{proof}
The receipt verifier checks all $39{,}672$ terminal receipts and their raw
outputs.  The coverage verifier reconstructs the post-cross-distance
canonical fan-outs from the recorded metadata and proves a prefix-free,
complete trie cover for each of the nine solver modes representing the eight
configurations.  Every planned leaf has exact status \code{ZERO}, exit code
$0$, zero solutions, and no abandoned branch.  Consequently, at every split,
each canonical child is either covered by one of these terminal searches or
rejected by a necessary condition as described in
\cref{subsec:partitioning}.  No branch is closed merely by prefix-freeness,
and no accepting leaf or abandoned branch remains.
\end{proof}

\section{Proof of the main theorem}\label{sec:main-proof}

\begin{proof}[Proof of \Cref{thm:main}]
Suppose an order-$18$ Leech tree exists.  By
\Cref{thm:eight-configs}, its first three physical edges lie in one of
Configurations 1 through 8.  By \Cref{prop:search-complete}, the search rooted
at that configuration reaches an accepting leaf.  The filters used in the
corresponding search are necessary conditions, and fail-open resource policies
cannot remove a genuine completion.  Yet \Cref{prop:certified-exhaustion}
states that every canonical search branch in every configuration is exhausted
by a terminal search, closed by explicit exact-zero evidence, or rejected by a
proved necessary condition, with no accepting leaf or abandoned branch
remaining.  This is a contradiction.
\end{proof}

\section{Formal status and trust boundary}
\label{sec:formal-status}

\subsection{Trust boundary}\label{subsec:trust-boundary}

The Lean development is available at \leanartifact, tag \code{v1.0.0}.  The
computational artifact repository is available at \computationalartifact, tag
\code{v1.0.0}.  The exact evidence objects used in this paper are identified
by filename and \sha{} digest in \cref{app:reproducibility}; the repository
release record and manifest identify the packaged copy containing them.
The Lean development uses Lean~4 version~4.24.0 and Mathlib version~4.24.0 and
builds successfully.
The Mathlib commit is
\begin{center}
\small\code{f897ebcf72cd16f89ab4577d0c826cd14afaafc7}
\end{center}
It contains no \code{sorry}, \code{admit}, or \code{native_decide}.  The axiom
audit for the declarations corresponding to the Lean-verified rows of
\cref{tab:formal-status} uses only \code{propext},
\code{Classical.choice}, and \code{Quot.sound}.

\begin{table}[H]
\centering
\small
\begin{tabularx}{\textwidth}{@{}>{\raggedright\arraybackslash}X
                                  >{\raggedright\arraybackslash}X@{}}
\toprule
Ingredient & Verification status\\
\midrule
Taylor order restriction & Lean verified; Taylor's result retains its attribution \citep{Taylor1977}\\
Order-$18$ parity split & Lean verified\\
Hop-diameter bound & Lean verified; self-contained proof included\\
First physical weights & Lean verified\\
Forced least-missing distance & Lean verified\\
Persistent merge block & Lean verified\\
Eight-configuration exhaustiveness & Lean verified; proof included\\
Selected-edge gluing identity & Lean verified\\
Component-pair whole-block exact-cover theorem & Conventional proof; not Lean verified\\
Search-completeness proposition & Conventional proof; not Lean verified\\
C++ search and certificate generation & Computational result; not Lean verified\\
Certificate checking and global collection & Computational result; not Lean verified\\
\bottomrule
\end{tabularx}
\caption{Formal and computational status of the proof ingredients.}
\label{tab:formal-status}
\end{table}

\paragraph{Lean declaration map.}
\begin{description}[style=nextline,leftmargin=0pt,labelsep=0pt,itemsep=0.5em]
\item[First physical weights]
\small\nolinkurl{LeechTrees.Foundation.T1_physical_weights_one_two}
\item[Forced least-missing distance]
\small\nolinkurl{LeechTrees.Foundation.PosIntTree.t2_forced_mex}
\item[Persistent merge block]
\small\nolinkurl{LeechTrees.Foundation.T2_forced_mex_merge_persistence}
\item[Taylor order restriction]
\small\nolinkurl{LeechTrees.Foundation.t3_taylor_order_condition}
\item[Order-$18$ parity split]
\small\nolinkurl{LeechTrees.Foundation.t3_order18_class_sizes}

\small\nolinkurl{LeechTrees.Foundation.T3_taylor_parity_order18}
\item[Hop-diameter bound]
\small\nolinkurl{LeechTrees.QHop.order18_simplePath_length_le_14}
\item[Eight-configuration exhaustiveness]
\small\nolinkurl{LeechTrees.Foundation.FirstEdgeDossier.eightRowDossier_of_weights}

\small\nolinkurl{LeechTrees.Foundation.FirstEdgeDossier.firstEdge_eightRowDossier}
\item[Selected-edge gluing identity]
\small\nolinkurl{LeechTrees.PathMulticut.actual_selectedEdge_gluing_polynomial}
\end{description}

The component-pair whole-block theorem and the search-completeness proposition
are ordinary mathematical arguments in this paper.  The exhaustive search,
certificates, and certificate checkers are not Lean verified.  The final
result is therefore a computer-assisted proof, not an end-to-end Lean proof.

The trusted computational foundation consists of the frozen C++ generation of
children and distances, the implementation of the production filters, the
canonical codes and automorphism-orbit reduction, the recorded search
executions, and the programs that write and collect receipts.  The separate
verification scripts re-read receipts, totals, plans, digests, and recorded
fan-outs, but the coverage verifier obtains each fan-out from solver-generated
\code{child_max} data.  It does not reconstruct the canonical child set from
the mathematical state.  Canonical-child generation therefore remains within
the trusted foundation.  These checks are not an independent reimplementation
and are not kernel-checked computation.

\section{Future work}\label{sec:future}

The next objective is to reduce the trusted computational base.  Future work
will formalize \cref{thm:whole-block}, the prefix-state invariant, and the
forced recursive enumeration; verify the canonical-orbit reduction or replace
it with a certificate format that does not require trusting it; implement a
small Lean checker for the eight branch covers and final zero receipts; and
connect the kernel-checked structural reduction directly to the archived
certificates.

\appendix
\section{Reproducibility identifiers}\label[appendix]{app:reproducibility}

The clean artifact maps its nine solver modes to the eight mathematical
configurations as follows.

\begin{table}[htbp]
\centering
\small
\begin{tabularx}{\textwidth}{@{}l X X@{}}
\toprule
Configuration & Solver mode & Terminal receipts\\
\midrule
1 & \code{g001_row0} & 5,202\\
2 & \code{g001_row1} & 79\\
3 & \code{a2_attached}, \code{a2_separate} & $7+40$\\
4 & \code{g001_row3} & 1,324\\
5 & \code{g001_row4} & 25,684\\
6 & \code{g001_row5} & 3,983\\
7 & \code{g001_row6} & 3,301\\
8 & \code{g001_row7} & 52\\
\bottomrule
\end{tabularx}
\caption{Mode-to-configuration mapping in the clean release.}
\label{tab:execution-mapping}
\end{table}

The clean execution plan contains $39{,}672$ terminal records in $208$
bundles.  Of these, $39{,}178$ are labelled \code{primary_cover}, $464$ are
labelled \code{coverage_completion}, and $30$ are labelled
\code{direct_replay_completion}.  All are uncapped terminal searches with
exact result \code{ZERO}.  The global result records\\
\code{all_eight_configurations_exact_zero=true} and binds the eight
per-configuration results and the complete coverage result.

\subsection{Implementation validation}

The released rebuild check compiles all four included programs from the
archived C++ source and runs the $45$-check solver regression, the $11$-check
independent witness-checker test, the small-order test, and a known partition
that visits $32{,}841$ nodes.  These are implementation tests, not a second
proof of the order-$18$ result.  The included ELF binaries identify their
compiler as GCC 11.5.0 20240719 (Red Hat 11.5.0-5) on x86-64 Linux with
glibc 2.34; the exact C++20 compilation flags and binary hashes are recorded
in \path{runtime/runtime_manifest.json}.
\subsection{Reproducibility package and integrity checks}

A \sha{} digest identifies the exact bytes of a file.  Matching digests
establish file identity, not program correctness or search completeness.

Here a \emph{branch package} is the self-contained evidence for one part of
the search: its partition plan lists the subtree roots to be searched, its
terminal receipts record the result for every root, and its verification files
check the resulting evidence.  The word \emph{frozen} means that these inputs
and receipts were fixed before the final collection and verification rather
than edited afterward.  Prefix-freeness, coverage, and their logical
distinction are defined in \cref{subsec:partitioning}.  The \emph{global
record} collects the accepted branch results and records whether all eight
configurations have been closed.

The computational artifact repository is available at
\computationalartifact, tag \code{v1.0.0}.  The complete evidence is the
release asset
\begin{center}
\path{leech18_computational_evidence_v1.0.0.tar.gz},
\end{center}
whose size is $24{,}204{,}350$ bytes and whose \sha{} digest is
\begin{center}
\small\code{3b53ced3f8d0b8aff0798d446fdb6256e3ad688429df090a7f915849b59d23b8}.
\end{center}
This single archive contains the plan, source, frozen binaries, all $208$
bundle receipts, all $39{,}672$ terminal leaf receipts and raw outputs, the
eight configuration results, the coverage result, the global result, and the
verification program.  No configuration depends on a separate evidence
archive.

Within the archive, \path{MANIFEST.sha256} binds every other file and has
\sha{} digest
\begin{center}
\small\code{49c6bf74747c97d3a9da8defdecef289f83ca10f372ad66e3e61015b96c072cf}.
\end{center}
The clean execution plan has digest
\begin{center}
\small\code{c99a5deb2d22fa54367c66141ae2c456e301d01bbcc913c6fbd1c694eb02a62e},
\end{center}
the runtime manifest has digest
\begin{center}
\small\code{fee90418f103b7e3fdbf7afc4ce34ba043cb649204616fd0e940abc4ae3dbdb0},
\end{center}
the coverage result has digest
\begin{center}
\small\code{098083c68b020578e946ebee17fe909685e9f6d211c538ac301e89d4ec041ccf},
\end{center}
and the global result has digest
\begin{center}
\small\code{c2d2b6aca2868deed891f91f07f348a2f4c63403adf4754a74781684c8f0f017}.
\end{center}

\paragraph{Reproduction.}
After extracting the release archive, the complete eight-configuration
verifier is invoked from the artifact root by
\begin{lrbox}{\reproducecommandbox}
\ttfamily\verb|python3 verification/verify_release.py verify-release --root .|
\end{lrbox}
\begin{center}
\resizebox{\textwidth}{!}{\usebox{\reproducecommandbox}}
\end{center}
Its expected final marker is
\begin{center}
\texttt{\detokenize{CLEAN_RELEASE_VERIFY_OK records=39672 configurations=8}}
\end{center}
This command checks the manifest, privacy rules, plan and runtime bindings,
all bundles, all leaf receipts and raw outputs, exact-zero status, and
exhaustive prefix-tree coverage for all nine modes.  It checks the saved
evidence and does not repeat the $8.5$-billion-node search.  The separate
rebuild and single-leaf replay commands are documented in \path{VERIFY.md}.
Successful digest checks establish file identity, while successful
verification checks the stated plan, coverage, receipt, and package
conditions.  Neither establishes program correctness nor removes the
trusted-code limitations described in \cref{subsec:trust-boundary}.

\section*{Declaration of AI assistance}

Generative AI tools made substantive contributions to the research process. They assisted with mathematical exploration and the development and checking of arguments; the design, implementation, and review of the exhaustive search and its verification workflow; the analysis of computational results; and the organization, editing, and LaTeX preparation of the manuscript. AI-generated outputs were not accepted as evidence on their own. The claims in this paper rest on the mathematical arguments presented here, the Lean-verified results, and the released computational evidence. The author made the final research decisions, reviewed and approved the work, and accepts full responsibility for every claim in the paper.

\begingroup
\footnotesize
\setlength{\bibsep}{1pt}
\bibliographystyle{abbrvnat}
\bibliography{references}
\endgroup

\end{document}